\RequirePackage{fix-cm}
\documentclass[smallextended]{svjour3}       
\smartqed  
\usepackage{graphicx}
\usepackage{subfigure}
\usepackage{graphicx}
\usepackage{color} 
\usepackage{lineno,hyperref}
\RequirePackage{amsmath,amsfonts}
\usepackage{float}
\usepackage{geometry}
\usepackage{epstopdf} 
\usepackage{latexsym}
\usepackage{mathptmx}
\newcommand{\cc}{\mathbb C}
\newcommand{\ee}{\mathbb E}
\newcommand{\pp}{\mathbb P}
\newcommand{\nn}{\mathbb N}
\newcommand{\rr}{\mathbb R}
\newcommand{\zz}{\mathbb Z}

\newtheorem{assumption}{Assumption}

\newcommand{\CC}{\mathcal C}
\newcommand{\LL}{\mathcal L}

\newcommand{\OOO}{\mathscr O}
\newcommand{\FFF}{\mathscr F}

\usepackage{mathrsfs}
\usepackage{mathptmx} 
\usepackage{amssymb}

\allowdisplaybreaks \allowdisplaybreaks[4]
\begin{document}

\title{Strong convergence of numerical discretizations for semilinear stochastic evolution equations driven by multiplicative white noise\thanks{Authors are supported by National Natural Science Foundation of China (NO. 91530118, NO. 91130003, NO. 11021101, NO. 91630312, and NO. 11290142).}}

\titlerunning{Strong convergence of numerical discretizations for semilinear SEEs}        

\author{Jialin Hong         \and
        Chuying Huang \and 
        Zhihui Liu 
}


\institute{Jialin Hong
	          \at
              LSEC, ICMSEC, Academy of Mathematics and Systems Science, Chinese Academy of Sciences, Beijing 100190, China;\\
              School of Mathematical Sciences, University of Chinese Academy of Sciences, Beijing 100049, China\\
              \email{hjl@lsec.cc.ac.cn}         
           \and
           Chuying Huang (Corresponding author) 
           \at
           LSEC, ICMSEC, Academy of Mathematics and Systems Science, Chinese Academy of Sciences, Beijing 100190, China;\\
           School of Mathematical Sciences, University of Chinese Academy of Sciences, Beijing 100049, China\\
           \email{huangchuying@lsec.cc.ac.cn}
           \and
           Zhihui Liu  
           \at
           Department of Mathematics,
           The Hong Kong University of Science and Technology,
           Clear Water Bay, Kowloon, Hong Kong\\
           \email{zhliu@ust.hk}
}

\date{Received: date / Accepted: date}

\maketitle

\begin{abstract}
For semilinear stochastic evolution equations whose coefficients are more general than the classical global Lipschitz, we present results on the strong convergence rates of numerical discretizations. The proof of them provides a new approach to strong convergence analysis of numerical discretizations for a large family of second order parabolic stochastic partial differential equations driven by space-time white noises.
We apply these results to the stochastic advection-diffusion-reaction equation with a gradient term and multiplicative white noise, and show that the strong convergence rate of 
a fully discrete scheme constructed by spectral Galerkin approximation and explicit exponential integrator is exactly $\frac12$ in space and $\frac14$ in time.
Compared with the optimal regularity of the mild solution, it indicates that the spetral Galerkin approximation is superconvergent and the convergence rate of the exponential integrator is optimal.
Numerical experiments support our theoretical analysis.
\end{abstract}

\section{Introduction}
\label{sec1}

Consider the semilinear stochastic evolution equation in the form
\begin{align}\label{spde} \tag{SEE}
	\begin{split}
		&{\rm d}X(t)=(AX(t)+F(X(t))) {\rm d}t+G(X(t)) {\rm d}W(t),
		\ t\in (0,T];\\
		& X(0)=X_0,
	\end{split}
\end{align}
on a Hilbert space $H$. Here $T$ is a fixed positive number, $A$ is a generator of an analytic $C_0$-semigroup $S$ on $H$ and $W=\{W(t):\ t\in [0,T]\}$ is a (possibly cylindrical) ${\bf Q}$-Wiener process. 
Since the exact solutions of stochastic partial differential equations (SPDEs) can not be formulated explicitly in general, there have been numerous works on numerical approximations for SPDEs; see e.g., \cite{ACQ17,BJK16(SINUM),CHL17(SINUM),CHL17(JDE),DG01(MOC),Gyo98(PA),Gyo99(PA),Kru14(IMA),Wal05(PA)} and references therein. 
One attractive and challenging topic is how to prove the critical strong convergence rate of certain numerical approximation while relaxing the conditions on coefficients.
The classical globally Lipschitz condition reads that
\begin{align}\label{globalLip}
	\|F(x)-F(y)\|\le C\|x-y\| 	\quad \text{and}\quad  \|G(x)-G(y)\|_{\LL_2^0}\le C\|x-y\|,\quad \forall~ x,y\in H,
\end{align}
where $C$ is a positive constant, $\LL^{0}_2$ denotes the space of Hilbert-Schmidt operators $\Phi:{\bf Q}^{\frac12}(U)\rightarrow H$ with the norm 
\begin{align*}
	\|\Phi\|_{\LL_2^0}:=\sum_{j=1}^{\infty}\|\Phi e_j\|^2
\end{align*}
and $\{e_j\}_{j\in\mathbb{N}_+}$ is an orthonormal basis of $U_0:={\bf Q}^{\frac12}(U)$. For second order parabolic SPDEs such as the stochastic heat equation, if the noise is the multiplicative space-time white noise (${\bf Q}=I$), condition \eqref{globalLip} is not fulfilled in general, even though the operators $F$ and $G$ are defined by the Nemytskii operators associated to globally Lipschitz continuous functions mapping from $\mathbb{R}$ to $\rr$. Furthermore, the strong convergence rates of numerical approximations in this case have not been studied thoroughly.
In this paper, based on \cite{HHL19JDE,HL19JDE},
we present a new approach to deriving strong convergence rates for numerical discretizations under assumptions which are weaker than condition \eqref{globalLip}. The assumptions are proposed in combination of the semigroup generated by the linear operator and the smoothing effect of the semigroup is more naturally exploited. Therefore, the approach is applicable to the strong convergence analysis of numerical discretizations for a large family of second order parabolic SPDEs, inclusive of the stochastic heat equation and the stochastic advection-diffusion-reaction equation where there is a gradient term describing the advection phenomenon \cite{BCD06(RG)}.
We highlight that the covariance operator ${\bf Q}$ of the Wiener process under study is not supposed to be of trace class. Rougher noises, including the multiplicative space-time white noise, are also available.

More precisely, we first consider a sequence of finite dimensional perturbation equations served as spatial discretizations for Eq. \eqref{spde}. We propose the assumptions with several parameters on the original equation and the perturbation equations, which are relevant to the regularity of their exact solutions. Then 
the spatial strong convergence rate of the sequence of perturbation equations is obtained by means of the Gronwall's inequality with singular kernel.
We further employ the explicit exponential integrator to construct a fully discrete scheme, 
which has implementation advantages over implicit schemes and avoids the CFL-type condition which appears in most of explicit schemes for SPDEs such as Euler--Maruyama scheme (see also \cite{ACLW16(SINUM),AC18(JCM),CLS13(SINUM),CD17(SPDE),Wan15(JSC)}). Combining with a temporally uniform H\"older continuity condition on the solutions of perturbation equations, we prove the strong convergence rate in the temporal direction. 

In order to specify the parameters in the proposed assumptions, we interpret our approach by applying it to a second order parabolic SPDE with a gradient term and multiplicative white noise. A typical example is the stochastic advection-diffusion-reaction equation. In \cite{CH12INT,CN13NM}, the authors obtain that the convergence rate in the H\"older norm of the Galerkin approximation is $\frac12-\epsilon$ and the convergence rate of the implicit-linear Euler scheme  is $\frac14-\epsilon$ in time, where $\epsilon$ denotes an infinitesimal factor.
As to smoother noises, 
\cite{LT13(IMA)} analyzes the exponential integrator for finite element discretizations of Eq. \eqref{spde} driven by trace class noises. 
In \cite{ZTRK15(SINUM)}, the authors study the stochastic collocation methods for linear advection-diffusion-reaction equations with finite dimensional Brownian motion.
In this paper, the spatial and temporal strong convergence order for a fully discrete scheme \eqref{scheme}, which is of spectral Galerkin approximation in space and explicit exponential integrator in time, achieves $\frac12$ and $\frac14$, respectively (see Theorem \ref{tm-whi}).
Compared to the optimal regularity of the solution illustrated in Remark \ref{rk-whi}, the spectral Galerkin approximation is shown to be superconvergent and the temporal convergence rate of the exponential integrator is optimal.


The rest of this paper is structured as follows.
We derive a quantitative continuous dependence result for a sequence of perturbation equations of Eq. \eqref{spde} in the next section which contains the strong convergence rate of spatial approximation for Eq. \eqref{spde}.
In Section \ref{sec3}, we obtain the strong convergence rate of a fully discrete scheme under a temporally uniform H\"older continuity condition for the Galerkin approximated solutions.
Concrete examples are given in Section \ref{sec4}, where we verify the key assumptions to derive the strong convergence rate of the scheme under study.
We give several numerical experiments to confirm our theoretical results in the last section.

\section{Perturbation Equation}
\label{sec2}

In this section, we begin with introducing the setting of the semilinear stochastic evolution equations under study. Then we analyze a sequence of perturbation equations for the original equation \eqref{spde} and give the rate of strong convergence.

\subsection{Setting}
\label{sec2.1}

Let $(H,\|\cdot\|)$ be a separable Hilbert space and denote by $(\LL(H),\|\cdot\|_{\LL(H)})$ the space of bounded linear operators mapping from $H$ to $H$. 
We assume that the linear operator $A: D(A)\subseteq H\rightarrow H$ is the infinitesimal generator of an analytic $C_0$-semigroup $S$ and the resolvent set of $A$ contains all $\lambda\in \cc$ with $\Re [\lambda]\ge 0$. 
Then for $\theta\in \rr$, the fractional powers $(-A)^\theta$  of the operator $-A$ are well-defined.
We denote by $\dot{H}^\theta$ the domain of $(-A)^{\theta/2}$ equipped with the norm $\|\cdot\|_\theta:=\|(-A)^\frac\theta2 (\cdot)\|$. In particular, $\dot{H}^0=H$.
It is well known that (see, e.g., \cite[Chapter 2.6]{Paz83}) 
\begin{align}\label{ana}
&\|(-A)^\nu \|_{\LL(H)}\le C, \nonumber\\
&\|(-A)^\mu S(t)\|_{\LL(H)}\le C t^{-\mu},  \\
&\|(-A)^{-\beta} (S(t)-{\rm Id}_{H}) \|_{\LL(H)}\le Ct^\beta, \nonumber
\end{align}
for any $0<t\le T$, $\nu \le 0\le \mu$ and $0\le \beta\le 1$.
Throughout this paper, $C$ is a generic positive constant independent of dimensions of the spatial discretization and the temporal step size in the sequel.

Let ${\bf Q}$ be a self-adjoint, nonnegative definite and bounded linear operator on a separable Hilbert space $U$. 
Define $\LL^{\theta}_2$ to be the set of Hilbert-Schmidt operators from $U_0:={\bf Q}^{\frac12}(U)$ to $\dot{H}^\theta$. By $W:=\{W(t):\ t\in [0,T]\}$, we denote a $U$-valued (possibly cylindrical) ${\bf Q}$-Wiener process with respect to a stochastic basis $(\Omega,\FFF,(\FFF_t)_{t\in [0,T]},\pp)$. 

Suppose that $F:H\rightarrow \dot{H}^{\theta_F}$, $\theta_F>-2$, and $G:H\rightarrow \LL(U,H)$.
Recall that a predictable stochastic process $X:[0,T]\times \Omega\rightarrow H$  is called a mild solution of Eq. \eqref{spde} if it holds a.s. that $X\in \mathbb L^\infty(0,T;H)$ and 
\begin{align}\label{mild}
X(t)=S(t)X_0+S*F(X)(t)+S\diamond G(X)(t), \quad t\in [0,T],
\end{align}
where $S*F(X)$ and $S\diamond G(X)$ denote the deterministic and stochastic convolutions, respectively:
\begin{align*}
&S*F(X)(\cdot):=\int_0^\cdot S(\cdot-r) F(X(r)){\rm d}r, \\
&S\diamond G(X)(\cdot):=\int_0^\cdot S(\cdot-r) G(X(r)){\rm d}W(r).
\end{align*}

To ensure the well-posedness of Eq. \eqref{spde}, we give the following assumption on the data of this equation in terms of \cite[Theorem 2.1]{HL19JDE}; see also \cite[Theorem 3.1]{HHL19JDE}. 
In the sequel, we use $\theta$ to denote a nonnegative number which characterizes the spatial regularity for the solution of Eq. \eqref{spde}.

\begin{assumption} \label{ap-con}
	There exist four nonnegative, Borel measurable functions $K_F,K_{F_\theta},K_G$ and $K_{G_\theta}$ on $[0,T]$ with
	\begin{align*}
	K^*(T): =\int_0^T \big[K_F(t) + K^2_G(t) \big] {\rm d}t<\infty,\quad
	K_{\theta}^* (T) :=\int_0^T \big[ K_{F_\theta} (t) + K^2_{G_\theta}(t) \big] {\rm d}t<\infty,
	\end{align*}
	such that for any $x,y\in H$, $z\in \dot{H}^\theta$ and almost every  $t\in[0,T]$, it holds that
	\begin{align*}
	\|S(t) (F(x)-F(y))\| &\le K_F(t)\|x-y\|,\\
	\|S(t) (G(x)-G(y))\|_{\LL^0_2} &\le K_G(t)\|x-y\|,\\
	\|S(t) F(z)\|_\theta &\le K_{F_\theta} (t)(1+\|z\|_\theta),\\ 
	\|S(t) G(z)\|_{\LL^{\theta}_2} &\le K_{G_\theta}(t)(1+\|z\|_\theta). 
	\end{align*}
\end{assumption}

\begin{lemma} (\cite[Theorem 2.1]{HL19JDE}) \label{lm-well}
	Let $p\ge 2$  and  $X_0:\Omega\rightarrow \dot{H}^\theta$ be   $\FFF_0/\mathcal{B}(\dot{H}^\theta)$-measurable such that $X_0\in \mathbb L^p(\Omega;\dot{H}^\theta)$. If Assumption \ref{ap-con} holds, then
	Eq. \eqref{spde} admits a unique mild solution $X=\{X(t):\ t\in [0,T]\}$ and there exists a constant $C=C(T,p,K^*_{\theta}(T))$ such that
	\begin{align}\label{well}
	\sup_{t\in [0,T]}\ee\Big[\|X(t)\|_{\theta}^p\Big]
	&\le C\Big(1+\ee\Big[\|X_0\|_{\theta}^p \Big]\Big).
	\end{align}
\end{lemma}

\subsection{Perturbation Equation}

In this subsection,
we consider a family of perturbation equations which are deduced by certain spatial semidiscretization for Eq. \eqref{spde},
\begin{align}\label{spdeG} \tag{SEEN}
\begin{split}
&{\rm d}X^N(t)=(A^NX^N(t)+F^N(X^N(t))) {\rm d}t+G^N(X^N(t)) {\rm d}W(t),
\ t\in (0,T];\\
& X^N(0)=X^N_0.
\end{split}
\end{align}
For any $N\in \nn_+$, the linear operator $A^N$ is the infinitesimal generator of an analytic $C_0$-semigroup $S^N$ on $H$.
For example, the perturbation equations are constructed by the Galerkin method in Section \ref{sec4}.

\begin{assumption} \label{ap-spe}
	There exist four nonnegative, Borel measurable functions $K_{F^N},K_{F^N_\theta},K_{G^N}$ and $K_{G^N_\theta}$ on $[0,T]$ with
	\begin{align*}
	\widetilde{K^*}(T):=\sup_{N\in \nn_+}\int_0^T K_{F^N}(t) {\rm d}t+\sup_{N\in \nn_+}\int_0^T K^2_{G^N}(t) {\rm d}t<\infty, \\
	\widetilde{K_{\theta}^*} (T):=\sup_{N\in \nn_+}\int_0^T K_{F^N_\theta} (t) {\rm d}t+\sup_{N\in \nn_+}\int_0^T K^2_{G^N_\theta}(t) {\rm d}t<\infty,
	\end{align*}
	such that for any $x,y\in H$, $z\in \dot{H}^\theta$ and almost every $t\in[0,T]$, it holds that
	\begin{align*}
	\|S^N(t) (F^N(x)-F^N(y))\| &\le K_{F^N}(t)\|x-y\|, \\
	\|S^N(t) (G^N(x)-G^N(y))\|_{\LL^0_2} &\le K_{G^N}(t)\|x-y\|, \\
	\|S^N(t) F^N(z)\|_\theta &\le K_{F^N_\theta} (t)(1+\|z\|_\theta), \\
	\|S^N(t) G^N(z)\|_{\LL^{\theta}_2} &\le K_{G^N_\theta}(t)(1+\|z\|_\theta). 
	\end{align*}
\end{assumption}

Similar to the case of equation \eqref{spde}, the well-posedness of the perturbation equation \eqref{spdeG} is ensured by Assumption \ref{ap-spe}.

\begin{lemma}\label{lm31}
	Let $p\ge 2$ and $X^N_0:\Omega\rightarrow \dot{H}^\theta$ be   $\FFF_0/\mathcal{B}(\dot{H}^\theta)$-measurable and uniformly bounded in $\mathbb L^p(\Omega;\dot{H}^\theta)$. If Assumption \ref{ap-spe} holds, 
	then Eq. \eqref{spdeG} admits a unique mild solution
	$X^N=\{X^N(t):\ t\in [0,T]\}$, and there exists a constant $C=C(T,p,\widetilde{K_\theta^*}(T))$ such that 
	\begin{align}\label{welln}
	\sup_{N\in \nn_+}\sup_{t\in [0,T]}\ee\Big[\|X^N(t)\|_{\theta}^p\Big]
	&\le C\Big(1+\sup_{N\in \nn_+}\ee\Big[\|X^N_0\|_{\theta}^p \Big]\Big).
	\end{align}
\end{lemma}

In order to obtain the strong convergence rate of the perturbation equations for the original equation, we propose the two assumptions, in which the parameters $r_0,r_F,r_G$ will be concretized by a concrete example in Section \ref{sec4}.

\begin{assumption} \label{ap-spe-X0}
	The initial values $X_0,X^N_0:\Omega\rightarrow \dot{H}^\beta$ are $\FFF_0/\mathcal{B}(\dot{H}^\beta)$-measurable such that $X_0,X^N_0\in \mathbb L^p(\Omega;\dot{H}^\beta)$ and there exists a positive number $r_0$ such that 
	\begin{align}
	\sup_{t\in [0,T]}\|S^N(t)X^N_0-S(t)X_0\|_{\mathbb L^p(\Omega;H)}
	\le CN^{-r_0}\|X_0\|_{\mathbb L^p(\Omega;\dot{H}^\beta)}.
	\end{align}
\end{assumption}

\begin{assumption} \label{ap-spe-ord}
	There exist two nonnegative, Borel measurable functions $R_{F^N},R_{G^N}$ on $[0,T]$ and two positive numbers $r_F,r_G$ depending on $\theta$ with
	\begin{align*}
	\int_0^T R_{F^N}(t) {\rm d}t\le CN^{-r_F}
	\quad \text{and}\quad 
	\left(\int_0^T R^2_{G^N}(t) {\rm d}t\right)^\frac12\le CN^{-r_G},
	\end{align*}
	such that for any $z\in \dot{H}^{\theta}$ and almost every $t\in[0,T]$, it holds that
	\begin{align*}
	\|S^N(t)F^N(z)-S(t)F(z)\|&\le R_{F^N}(t)(1+\|z\|_{\theta}),\\
	\|S^N(t)G^N(z)-S(t)G(z)\|_{\LL^0_2}&\le R_{G^N}(t)(1+\|z\|_{\theta}).
	\end{align*}
\end{assumption}

Our main result of this section is on the strong error between the perturbation equation \eqref{spdeG} and the original equation \eqref{spde}.

\begin{theorem}\label{x-xn}
	Let $p\ge 2$ and $\beta\ge\theta\ge 0$. If Assumptions \ref{ap-con}--\ref{ap-spe-ord} hold,
	then there exists a constant $C$ such that 
	\begin{align}\label{x-xn0}
	\sup_{t\in [0,T]}\|X^N(t)-X(t)\|_{\mathbb L^p(\Omega;H)}\le C N^{-(r_0 \wedge r_F \wedge r_G)}.
	\end{align}
	
\end{theorem}

\begin{proof}
	For any $t\in [0,T]$, 
	we decompose the error between $X^N(t)$ and $X(t)$ as
	\begin{align*}
	\left\|X^N(t)-X(t)\right\|_{\mathbb L^p(\Omega;H)}
	\le& \left\|S^N(t)X^N_0-S(t)X_0\right\|_{\mathbb L^p(\Omega;H)}\\
	&+\left\|\int_0^t \left(S^N(t-r) F^N(X^N(r))-S(t-r) F(X(r))\right){\rm d}r\right\|_{\mathbb L^p(\Omega;H)}\\
	&+\left\|\int_0^t \left(S^N(t-r) G^N(X^N(r))
	- S(t-r) G(X(r))\right){\rm d}W(r)\right\|_{\mathbb L^p(\Omega;H)}\\
	= &:I_1(t)+I_2(t)+I_3(t).
	\end{align*}
	
	For the first term, it follows from Assumption \ref{ap-spe-X0} that 
	\begin{align*}
	I_1(t)\le CN^{-r_0}\|X_0\|_{\mathbb L^p(\Omega;\dot{H}^\beta)}.
	\end{align*}
	We deduce by the Minkowski inequality and Assumptions \ref{ap-spe}-\ref{ap-spe-ord} that
	\begin{align*}
	I_2(t)
	&\le \int_0^t \|S^N(t-r) F^N(X^N(r))-S^N(t-r) F^N(X(r))\|_{\mathbb L^p(\Omega;H)}{\rm d}r\\
	& \quad +\int_0^t \|S^N(t-r) F^N(X(r))-S(t-r) F(X(r)\|_{\mathbb L^p(\Omega;H)}{\rm d}r\\
	&\le \int_0^t K_{F^N}(t-r)\|X^N(r)-X(r)\|_{\mathbb L^p(\Omega;H)}{\rm d}r\\
	& \quad +\int_0^t R_{F^N}(t-r)\left(1+\|X(r)\|_{\mathbb L^p(\Omega;\dot{H}^\theta)}\right){\rm d}r.
	\end{align*}
	Moreover, the Burkholder--Davis--Gundy inequality and Assumptions \ref{ap-spe}-\ref{ap-spe-ord} yield
	\begin{align*}
	I_3^2(t)\le&
	2\int_0^t \|S^N(t-r) G^N(X^N(r))-S^N(t-r) G^N(X(r))\|^2_{\mathbb L^p(\Omega;\LL^0_2)}{\rm d}r\\
	&+2\int_0^t\|S^N(t-r) G^N(X(r))-S(t-r) G(X(r))\|^2_{\mathbb L^p(\Omega;\LL^0_2)}{\rm d}r\\
	\le& 2\int_0^t K^2_{G^N}(t-r)\|X^N(r)-X(r)\|^2_{\mathbb L^p(\Omega;H)} {\rm d}r\\
	&+2\int_0^t R^2_{G^N_\theta}(t-r)\left(1+\|X(r)\|_{\mathbb L^p(\Omega;\dot{H}^\theta)}\right)^2{\rm d}r.
	\end{align*}
	
	Together with the previous estimations with the Minkowski and H\"older inequalities, we derive
	\begin{align*}
	\|X^N(t)-X(t)\|^2_{\mathbb L^p(\Omega;H)}
	\le& CN^{-2r_0}\|X_0\|^2_{\mathbb L^p(\Omega;\dot{H}^\beta)}\\
	& +C\bigg(1+\sup_{t\in [0,T]}\|X(t)\|_{\mathbb L^p(\Omega;\dot{H}^\theta)}\bigg)^2\bigg(\int_0^T R_{F^N}(t){\rm d}t\bigg)^2\\
	& +C\bigg(1+\sup_{t\in [0,T]}\|X(t)\|_{\mathbb L^p(\Omega;\dot{H}^\theta)}\bigg)^2\bigg(\int_0^T R_{G^N}^2(t){\rm d}t\bigg)\\
	& +C
	\int_0^t \widetilde{K^*}(T) K_{F^N}(t-r)\|X^N(r)-X(r)\|^2_{\mathbb L^p(\Omega;H)}{\rm d}r\\
	& +C\int_0^t K^2_{G^N}(t-r)\|X^N(r)-X(r)\|^2_{\mathbb L^p(\Omega;H)} {\rm d}r.
	\end{align*}
	Define 
	$M_N(t):=\widetilde{K^*}(T)K_{F^N}(t)+K_{G^N}^2(t)$, $t\in [0,T]$.
	According to Assumption \ref{ap-spe}, it is clear that $M_N$ is uniformly integrable on $[0,T]$, $N\in\nn_+$. 
	By Lemmas \ref{lm-well}-\ref{lm31} and Assumption \ref{ap-spe-ord}, we obtain 
	\begin{align*}
	\|X^N(t)-X(t)\|^2_{\mathbb L^p(\Omega;H)}
	\le C N^{-2(r_0 \wedge r_F \wedge r_G)}+C\int_0^t M_N(t-r)\|X^N(r)-X(r)\|^2_{\mathbb L^p(\Omega;H)} {\rm d}r,
	\end{align*}
	from which we conclude \eqref{x-xn0} by the Gronwall's inequality with singular kernel in \cite[Lemma 3.1]{HL19JDE}.
\end{proof}
\section{Fully Discrete Scheme}
\label{sec3}
To construct a fully discretization of Eq. \eqref{spde}, we apply the explicit exponential integrator to the perturbation equation \eqref{spdeG}. More precisely, given $K\in \nn_+$, we define approximation $X^N_{k}$ for $X^N(t_k)$ by the recursion
\begin{align}\label{scheme} \tag{EI}
X^N_{k}=S^N(\tau)X^N_{k-1}+\tau S^N(\tau)F^N(X^N_{k-1})+S^N(\tau)G^N(X^N_{k-1})\Delta W_{k}, \quad k=1,\cdots,K
\end{align}
with the same initial datum as Eq. \eqref{spdeG}. Here $\tau=\frac{T}{K}$, $t_k=k\tau$ and $\Delta W_{k}=W(t_{k})-W(t_{k-1})$.
Equivalently, 
\begin{align*}
X^N_k= S^N(t_k)X^N_0+\sum_{i=0}^{k-1}\int_{t_i}^{t_{i+1}}S^N(t_k-t_i)F^N(X^N_i)  {\rm d}r +\sum_{i=0}^{k-1}\int_{t_i}^{t_{i+1}}S^N(t_k-t_i)G^N(X^N_i)  {\rm d} W(r).
\end{align*}


To derive the strong convergence rate of scheme \eqref{scheme}, we propose the following assumption on $S^N, F^N$ and $G^N$.

\begin{assumption}\label{ap-tem}
	There exist two nonnegative, Borel measurable functions $R_{F^{N,\tau}}$ and $R_{G^{N,\tau}}$ on $[0,T]$ and two positive numbers $\eta_F$, $\eta_G$ depending on $\theta$ with
	\begin{align*}
	\sup_{N\in \nn_+}\int_0^{T}R_{F^{N,\tau}}(t){\rm d}t\le C \tau^{\eta_F}
	\quad \text{and}\quad 
	\left(\sup_{N\in \nn_+}\int_0^{T}R^2_{G^{N,\tau}}(t){\rm d}t\right)^\frac12\le C \tau^{\eta_G},
	\end{align*}
	such that for any $z\in \dot{H}^{\theta}$ and almost every $t\in[0,T]$, it holds that
	\begin{align*}
	\|(S^N(t)-S^N(\lceil t \rceil_\tau))F^N(z)\|
	&\le R_{F^{N,\tau}}(t)(1+\|z\|_\theta), \\
	\|(S^N(t)-S^N(\lceil t \rceil_\tau))G^N(z)\|_{\LL^0_2}
	&\le R_{G^{N,\tau}}(t)(1+\|z\|_\theta).
	\end{align*}
	Here $\lceil t \rceil_\tau:=\min \{t_i:t_i\ge t,i=0,\cdots,K\}$.
\end{assumption}

In addition to the above assumption on the data of the perturbation equation \eqref{spdeG}, we  need a uniform H\"older regularity of their solutions.

\begin{assumption}\label{ap-xn-hol}
	The uniform H\"older exponent of $\{X^N(t):\ t\in [0,T]\}_{N\in \nn_+}$ in $\mathbb L^p(\Omega;H)$ is $\gamma$ for some $\gamma\in(0,1/2]$, i.e., 
	\begin{align*}
	\sup_{N\in \nn_+}\sup_{0\le s<t\le T}\frac{\|X^N(t)-X^N(s)\|_{\mathbb L^p(\Omega;H)}}{(t-s)^\gamma} <\infty .
	\end{align*}
\end{assumption}

\begin{remark}
	The H\"older exponent in Assumption \ref{ap-xn-hol} is not larger than $1/2$ due to the temporal regularity of $\bf Q$-Wiener processes. In the general multiplicative noise case, the uniform H\"older continuity assumption is essential in our error analysis.
	We also note that in the additive noise case, for instance, if there exists a constant operator $G$ such that $G(z)\equiv G$ for all $z\in H$, then Assumption \ref{ap-xn-hol} is not necessary in the strong error analysis. 
\end{remark}

We also need a discrete version of the Gronwall's inequality with singular kernel. For the corresponding continuous version, we refer to \cite[Lemma 3.1]{HL19JDE}.

\begin{lemma}\label{Gron}
	Let $m>0$ and $\Psi^N: [0,T] \rightarrow \mathbb{R}$, $N\in \nn_+$, be a sequence of nonnegative, Borel measurable functions such that
	\begin{align}\label{Gron_R}
	\alpha(T):=\sup_{N\in \nn_+} \sup_{K\in \nn_+}\sum_{i=1}^{K} \Psi^N(t_i)\tau<\infty.
	\end{align}
	For any $N\in \nn_+$, assume that $\{f^N(k)\}_{k=0}^{K}$ is a nonnegative sequence such that 
	\begin{align*}
	f^N(k)\le m + \sum_{i=0}^{k-1} \Psi^N(t_k-t_i)f^N(i)\tau,
	\end{align*}
	then there exists a constant $\mu$ independent of $K$ and $N$ such that 
	\begin{align*}
	\sup_{N\in \nn_+}  \sup_{K\in \nn_+} \sup_{0\le k\le K} f^N(k)\le 2me^{\mu T}.
	\end{align*}
\end{lemma}

\begin{proof}
	For any $\mu\ge 0$, it holds that 
	\begin{align*}
	e^{-\mu t_k}f^N(k)\le m + \sum_{i=0}^{k-1} e^{-\mu(t_k-t_i)}\Psi^N(t_k-t_i)e^{-\mu t_i}f^N(i)\tau.
	\end{align*}
	Denoting $f^N_\mu(k):=e^{-\mu t_k}f^N(k)$ and $\Psi^N_{\mu}(t):=e^{-\mu t}\Psi^N(t)$, we have
	\begin{align*}
	f^N_\mu(k)\le m + \sum_{i=0}^{k-1} \Psi^N_{\mu}(t_k-t_i)f^N_\mu(i)\tau.
	\end{align*}
	From condition \eqref{Gron_R}, 
	we get
	\begin{align*}
	\alpha_\mu(T):=\sup_{N\in \nn_+} \sup_{K\in \nn_+}\sum_{i=1}^{K} \Psi^N_{\mu}(t_i)\tau
	\end{align*}
	decreases with respect to $\mu$. Since
	\begin{align*}
	\lim_{\mu\rightarrow 0}\alpha_\mu(T)\le\alpha(T)
	\quad \text{and}\quad 
	\lim_{\mu\rightarrow \infty}\alpha_\mu(T)=0,
	\end{align*}
	there exists a constant $\mu_0$ independent of $K$ and $N$ such that 	
	$\alpha_{\mu_0}(T)\le\frac1{2}$.
	As a consequence,
	\begin{align*}
	f^N_{\mu_0}(k)
	\le m+ \alpha_{\mu_0}(T) \sup_{0\le i\le k-1}f^N_{\mu_0}(i) \le m+ \frac12\sup_{0\le i\le K}f^N_{\mu_0}(i),
	\end{align*}
	from which we obtain $\sup_{0\le k\le K}f^N_{\mu_0}(k)\le 2m$.
	This completes the proof.
\end{proof}

The next theorem shows the strong convergence rate of the fully discrete scheme for Eq. \eqref{spde}.

\begin{theorem} \label{main} 
	Let $p\ge 2$ and $\beta\ge\theta\ge 0$. If Assumptions \ref{ap-con}--\ref{ap-xn-hol} hold and the sequence of functions $K_{F^N}+K^2_{G^N}$ satisfies condition \eqref{Gron_R},
	then there exists a constant $C$ such that 
	\begin{align} \label{x-xnk}
	\sup_{0\le k \le K}\|X(t_k)-X^N_k\|_{\mathbb L^p(\Omega;H)}
	\le C (N^{-(r_0 \wedge r_F \wedge r_G)}+\tau^{\gamma \wedge \eta_F \wedge \eta_G}).
	\end{align}
\end{theorem}

\begin{proof}
	In terms of Theorem \ref{x-xn} and the triangle inequality, it suffices to show that 
	\begin{align} \label{xn-xnk}
	\sup_{0\le k \le K}\|X^N(t_k)-X^N_k\|_{\mathbb L^p(\Omega;H)}
	\le C \tau^{\gamma \wedge \eta_F \wedge \eta_G}.
	\end{align}
	
	For any $k=1,\cdots,K$, the error satisfies 
	\begin{align*}
	\left\|X^N(t_k)-X^N_k\right\|_{\mathbb L^p(\Omega;H)}\le &
	\left\| \sum_{i=0}^{k-1}\int_{t_i}^{t_{i+1}}\left(S^N(t_k-r)F^N(X^N(r))-S^N(t_k-t_i)F^N(X^N_i) \right){\rm d}r\right\|_{\mathbb L^p(\Omega;H)}\\
	&+\left\|\sum_{i=0}^{k-1}\int_{t_i}^{t_{i+1}}\left(S^N(t_k-r)G^N(X^N(r))-S^N(t_k-t_i)G^N(X^N_i)\right)  {\rm d} W(r)\right\|_{\mathbb L^p(\Omega;H)}\\
	=&:J_1+J_2.
	\end{align*}
	
	Using the Minkowski inequality, we have
	\begin{align*}
	J_1
	\le & \sum_{i=0}^{k-1}\int_{t_i}^{t_{i+1}} \|(S^N(t_k-r)-S^N(t_k-t_i))F^N(X^N(r))\|_{\mathbb L^p(\Omega;H)} {\rm d}r \\
	&+\sum_{i=0}^{k-1}\int_{t_i}^{t_{i+1}} \|S^N(t_k-t_i)(F^N(X^N(r))-F^N(X^N(t_i)))\|_{\mathbb L^p(\Omega;H)} {\rm d}r  \\
	&+\sum_{i=0}^{k-1}\int_{t_i}^{t_{i+1}} \|S^N(t_k-t_i)(F^N(X^N(t_i))-F^N(X^N_i))\|_{\mathbb L^p(\Omega;H)} {\rm d}r \\
	= &:J_{11}+J_{12}+J_{13}.
	\end{align*}
	For the term $J_{11}$, changing of variation and Assumptions \ref{ap-spe} and \ref{ap-tem} lead to
	\begin{align*}
	J_{11} 
	&\le  \sum_{i=0}^{k-1}\int_{t_i}^{t_{i+1}}\|(S^N(r)-S^N(\lceil t \rceil_\tau))F^N(X^N(t_k-r))\|_{\mathbb L^p(\Omega;H)}{\rm d}r\\
	&\le  \sum_{i=0}^{k-1}\int_{t_i}^{t_{i+1}}R_{F^{N,\tau}}(r)\left(1+\|X^N(t_k-r)\|_{\mathbb L^p(\Omega;\dot{H}^\theta)}\right){\rm d}r\\
	&\le  \left(\int_0^{t_k} R_{F^{N,\tau}}(r) {\rm d}r\right)
	\left(1+\sup_{t \in [0,T]}\|X^N(t)\|_{\mathbb L^p(\Omega;\dot{H}^\theta)}\right)
	\le C \tau^{\eta_F}.
	\end{align*}
	Assumptions \ref{ap-spe} and \ref{ap-xn-hol} imply
	\begin{align*}
	J_{12}
	&\le  \sum_{i=0}^{k-1}\int_{t_i}^{t_{i+1}}K_{F^N}(t_k-t_i)\|X^N(r)-X^N(t_i)\|_{\mathbb L^p(\Omega;H)}{\rm d}r \\
	&\le  C \tau^\gamma \bigg(\sum_{i=1}^{k}K_{F^N}(t_i)\tau\bigg).
	\end{align*}
	According to the H\"older inequality and Assumption \ref{ap-spe}, we get
	\begin{align*}
	J_{13} 
	&\le  \sum_{i=0}^{k-1} K_{F^N}(t_k-t_i)
	\|X^N(t_i)-X^N_i\|_{\mathbb L^p(\Omega;H)} \tau \\
	&\le  C\bigg(\sum_{i=1}^{k} K_{F^N}(t_i)\tau\bigg)^\frac12
	\bigg(\sum_{i=0}^{k-1} K_{F^N}(t_k-t_i)\|X^N(t_i)-X^N_i\|^2_{\mathbb L^p(\Omega;H)} \tau\bigg)^\frac12 \\
	&\le  C \bigg(\sum_{i=0}^{k-1} K_{F^N}(t_k-t_i)\|X^N(t_i)-X^N_i\|^2_{\mathbb L^p(\Omega;H)} \tau \bigg)^\frac12.
	\end{align*}
	Collecting the above estimations for $J_{11}$, $J_{12}$ and $J_{13}$ together, we have that
	\begin{align} \label{J1}
	J^2_1
	\le C \tau^{2(\gamma\wedge \eta_F)}
	+C\sum_{i=0}^{k-1} K_{F^N}(t_k-t_i)\|X^N(t_i)-X^N_i\|^2_{\mathbb L^p(\Omega;H)} \tau.
	\end{align}
	
	To estimate $J_2$, we utilize the Burkholder--Davis--Gundy inequality and similar arguments as in the estimations of $J_{11}$, $J_{12}$ and $J_{13}$ to show
	\begin{align*}
	J^2_2
	\le& C \sum_{i=0}^{k-1}\int_{t_i}^{t_{i+1}}
	\|(S^N(t_k-r)-S^N(t_k-t_i))G^N(X^N(r))\|^2_{\mathbb L^p(\Omega;\LL_2^0)} {\rm d}r \\
	&+C\sum_{i=0}^{k-1}\int_{t_i}^{t_{i+1}}
	\|S^N(t_k-t_i)(G^N(X^N(r))-G^N(X^N(t_i)))\|^2_{\mathbb L^p(\Omega;\LL_2^0)} {\rm d}r \\
	&+C\sum_{i=0}^{k-1}\int_{t_i}^{t_{i+1}}
	\|S^N(t_k-t_i)(G^N(X^N(t_i))-G^N(X^N_i))\|^2_{\mathbb L^p(\Omega;\LL_2^0)} {\rm d}r \\
	\le & C \bigg(\int_0^{t_k} R^2_{G^{N,\tau}}(r) {\rm d}r\bigg)
	\bigg(1+\sup_{t \in [0,T]}\|X^N(t)\|^2_{\mathbb L^p(\Omega;\dot{H}^\theta)}\bigg)  \\
	&+C \tau^{2\gamma} \bigg(\sum_{i=1}^{k}K^2_{G^N}(t_i) \tau\bigg)
	+C\sum_{i=0}^{k-1} K^2_{G^N}(t_k-t_i)\|X^N(t_i)-X^N_i\|^2_{\mathbb L^p(\Omega;H)} \tau.
	\end{align*}
	Therefore, it holds that
	\begin{align} \label{J2}
	J_2^2
	\le C \tau^{2(\gamma\wedge \eta_G)}
	+C\sum_{i=0}^{k-1} K^2_{G^N}(t_k-t_i)\|X^N(t_i)-X^N_i\|^2_{\mathbb L^p(\Omega;H)} \tau.
	\end{align}
	
	Combining \eqref{J1}-\eqref{J2}, we obtain
	\begin{align*}
	\|X^N(t_k)-X^N_k\|^2_{\mathbb L^p(\Omega;H)}\le C \tau^{2(\gamma \wedge \eta_F \wedge \eta_G)}+ 
	C \sum_{i=0}^{k-1} \Psi^N(t_k-t_i) 
	\|X^N(t_i)-X^N_i\|^2_{\mathbb L^p(\Omega;H)} \tau
	\end{align*}
	with $\Psi^N=K_{F^N}+K^2_{G^N}$, $N\in \nn_+$.
	As a consequence, we conclude \eqref{x-xnk} by Lemma \ref{Gron}.
\end{proof}


\section{Application}
\label{sec4}

In this section, we apply our results to the following second-order parabolic SPDE driven by the space-time white noise
\begin{align}\label{she}\tag{SHE}
\begin{split}
&{\rm d} X(t,\xi)
=(\Delta X(t,\xi)+\nabla f(X(t,\xi))+\widetilde f(X(t,\xi))) {\rm d}t +g(X(t,\xi)) {\rm d}W(t,\xi),\\
&X(t,\xi)=0, \quad (t,\xi)\in [0,T]\times \partial \OOO,\\
&X(0,\xi)=X_0(\xi), \quad \xi\in \OOO,
\end{split}
\end{align}
where
$\OOO\subset \rr$ is a bounded open set with Lipschitz boundary and ${\bf Q}={\rm Id}_H$, i.e., $W$ is the cylindrical Wiener process.
Moreover, it is assumed that $f,\widetilde f, g:\rr \rightarrow \rr$ are Lipschitz continuous functions with Lipschitz constants $L_f,L_{\widetilde f},L_g\ge 0$, i.e., for any $\xi_1$, $\xi_2\in \rr$,
\begin{align}\label{lip}
\begin{split}
|f(\xi_1)-f(\xi_2)|&\le L_f |\xi_1-\xi_2|, \\
|\widetilde f(\xi_1)-\widetilde f(\xi_2)|
&\le L_{\widetilde f} |\xi_1-\xi_2|,\\
|g(\xi_1)-g(\xi_2)|& \le L_g |\xi_1-\xi_2|.
\end{split}
\end{align}

Let $H=U=\mathbb L^2(\OOO)$ and define $A=\Delta$ with $\text{Dom}(A)=H_0^1(\OOO)\cap H^2(\OOO)$, where $H_0^1(\OOO):=\{f\in H^1(\OOO):\ f|_{\partial \OOO}=0\}$. 
It implies that the self-adjoint and positive definite operator $(-A)$ possesses an eigensystem $\{(\lambda_j,e_j)\}_{j\in \nn_+}$ with $\{\lambda_j\}_{j\in \nn_+}$ being an increasing sequence and $\{e_j\}_{j\in \nn_+}$ forming an orthonormal basis of $H$. 
It is known from Weyl's law (see, e.g., \cite{Gri02(CPDE)}) that 
\begin{align} \label{weyl}
\lambda_j\simeq j^2
\quad \text{and}\quad 
\|e_j\|_{\mathbb L^{\infty}(\OOO)}\le C.
\end{align} 
Define the operators $F_f:H\rightarrow \dot{H}^{-1},F_{\widetilde f}:H\rightarrow H$ and $G:H\rightarrow \LL(H)$ by Nemytskii operators associated with $\nabla f,\widetilde f$ and $g$, respectively: 
\begin{align*}
F_f(z)(\xi):=\nabla f(z(\xi)),\quad
F_{\widetilde f}(z)(\xi):=\widetilde f(z(\xi)),\quad
G(z)u(\xi):=g(z(\xi)) u(\xi),
\end{align*}
where $z\in H$, $\xi\in \OOO$ and $u\in H$.

Denoting $F(z):=F_f(z)+F_{\widetilde f}(z)$, we fit Eq. \eqref{she} into the framework considered in previous two sections with Assumption \ref{ap-con} for $\theta\in[0,1/2)$. Indeed, by the Lipschitz continuity of $f$ and $\widetilde f$, we have
\begin{align}
\|S(t) (F(x)-F(y))\|
\le& \|(-A)^\frac12 S(t) (-A)^{-\frac12} (F_f(x)-F_f(y))\|
+\|S(t) (F_{\widetilde f}(x)-F_{\widetilde f}(y))\| \nonumber\\
\le& \|(-A)^\frac{1}2 S(t)\|_{\LL(H)} \|F_{\widetilde f}(x)-F_{\widetilde f}(y)\|_{-1}
+\|S(t)\|_{\LL(H)} \|F_{\widetilde f}(x)-F_{\widetilde f}(y)\| \nonumber\\
\le& \left(L_f \|(-A)^{1/2} S(t)\|_{\LL(H)}+L_{\widetilde f}\|S(t)\|_{\LL(H)}\right)\|x-y\|\nonumber\\
\le &C(t^{-1/2}+1)\|x-y\|=: K_F(t)\|x-y\|.\label{kf0}
\end{align}
Similarly, we obtain that for any $z\in H$,
\begin{align}
\|S(t) F(z)\|_\theta 
\le& \|(-A)^\frac{1+\theta}2 S(t) (-A)^{-\frac12} F_f(z)\|
+\|(-A)^\frac\theta2 S(t) F_{\widetilde f}(z)\| \nonumber\\
\le& C(t^{-\frac{1+\theta}2}+t^{-\frac\theta2})(1+\|z\|)  =: K_{F_\theta}(t)(1+\|z\|).\label{kf}
\end{align}
Then the functions $K_F$ and $K_{F_\theta}$ are integrable on $[0,T]$. Concerning the diffusion term, since the Lipschitz continuity (and thus linear growth) of $g$ yields that
\begin{align*}  
\|S(t)(G(x)-G(y))\|^2_{\LL_2^0}
&=\sum_{j=1}^{\infty} \|S(t)(G(x)-G(y))e_j\|^2
\le C\sum_{j=1}^{\infty} e^{-2\lambda_j t} (1+\|x-y\|)^2
\end{align*}
and 
\begin{align*}
\|S(t)G(z)\|^2_{\LL^\theta_2}
&=\sum_{j=1}^{\infty}\lambda_j^\theta e^{-2\lambda_j t} \|G(z)e_j\|^2
\le C\sum_{j=1}^{\infty}\lambda_j^\theta e^{-2\lambda_j t} (1+\|z\|)^2,
\end{align*}
we define 
\begin{align} \label{kg-whi0}
K_G(t):=\left(C \sum_{j=1}^\infty e^{-2\lambda_j t}\right)^{\frac12}
\quad \text{and}\quad 
K_{G_\theta}(t):=\bigg(C \sum_{j=1}^{\infty}\lambda_j^\theta e^{-2\lambda_j t} \bigg)^\frac12.
\end{align} 
Moreover, we deduce
\begin{align*} 
\int_{0}^{T}K^2_G(t) {\rm d} t&\le C  \int_{0}^{T} t^{-\frac12} {\rm d} t\le C T^{\frac12},\\
\int_{0}^{T} K^2_{G_\theta}(t){\rm d} t& \le C \sum_{j=1}^{\infty}\lambda_j^\theta \left( \int_{0}^{T}  e^{-2\lambda_j t} {\rm d} t \right)\le C \sum_{j=1}^{\infty}\lambda_j^{\theta-1},
\end{align*} 
then $K_G$ and $K_{G_\theta}$ are square integrable on $[0,T]$ for any $\theta\in[0,\frac12)$.
By Lemma \ref{lm-well},	Eq. \eqref{she} possesses a unique mild solution $\{X(t):\ t\in [0,T]\}$ in $\CC([0,T]; \mathbb L^p(\Omega; \dot{H}^\theta))$ as soon as $X_0\in \mathbb L^p(\Omega;\dot{H}^\theta)$,  $p\ge 2$.

We use the spectral Galerkin method to construct the perturbation equation \eqref{spdeG}.
More precisely, for $N\in \nn_+$, let $V_N={\rm span}\{e_1,\cdots,e_N\}$ and $P_N$ be the projection from $H$ to $V_N$. 
Set $X^N_0=P_NX_0$, $A^N=AP_N$, $F^N=P_NF$ and $G^N=P_NG$, $N\in \nn_+$. 
Since $P_N$ is a contraction operator, one can take $K_{F^N}= K_F$, $K_{F^N_\theta}= K_{F_\theta}$, $K_{G^N}= K_G$ and $K_{G^N_\theta}= K_{G_\theta}$.
Thus Assumption \ref{ap-spe} holds and functions $\{K_{F^N}+K^2_{G^N}\}_{N\in\mathbb{N}_+}$ satisfies condition \eqref{Gron_R}.

To obtain the strong convergence rate of the fully discrete scheme \eqref{scheme}, we give the maximal values of $r_0$ in Assumption \ref{ap-spe-X0}, $r_F,r_G$ in Assumption \ref{ap-spe-ord}, $\eta_F,\eta_G$ in Assumption \ref{ap-tem},  and  $\gamma$ in Assumption \ref{ap-xn-hol}, respectively.

\begin{lemma} \label{ap-whi}
	Let $p\ge 2$ and $\beta\in (0,2]$. 
	Suppose that $X_0:\Omega\rightarrow \dot{H}^\beta$ is   $\FFF_0/\mathcal{B}(\dot{H}^\beta)$-measurable such that 
	$X_0\in \mathbb L^p(\Omega;\dot{H}^\beta)$.
	Then for Eq. \eqref{she}, Assumptions \ref{ap-spe-X0}-\ref{ap-tem}  hold  with
	\begin{align*}
	r_0=\beta, \quad 
	r_F=1-\epsilon,\quad\ r_G=\frac12, \quad
	\eta_F=\frac12-\epsilon,\quad \eta_G=\frac14,
	\end{align*}
	where $\epsilon$ is an arbitrary small positive number.
\end{lemma}

\begin{proof}
	Recall the standard estimation
	\begin{align*} 
	\|(P_N-{\rm Id}_H)z\|_{\nu}\le \lambda_{N+1}^{\frac{\nu-\mu}{2}}\|z\|_{\mu},
	\quad \forall~z\in \dot{H}^{\mu},\ \nu,\mu\in \rr.
	\end{align*}
	We know
	\begin{align*}
	\|S^N(t)X^N_0-S(t)X_0\|
	\le C \|(P_N-{\rm Id}_H)X_0\|
	\le C\lambda_{N+1}^{-\frac\beta2}\|X_0\|_\beta,
	\end{align*}
	and then $r_0=\beta$.
	
	The estimation \eqref{kf} shows that 
	\begin{align*}
	\|S^N(t)F^N(z)-S(t)F(z)\|
	&=\|(P_N-{\rm Id}_H) S(t) F(z)\|\\
	&\le C \lambda_{N+1}^{-\frac{\tilde{\theta}}{2}} \|S(t) F(z)\|_{\tilde{\theta}} \\
	&\le C \lambda_{N+1}^{-\frac{\tilde{\theta}}{2}} K_{F_{\tilde{\theta}}} (t) (1+\|z\|)
	\end{align*}
	and
	\begin{align*}
	\|(S^N(t)-S^N(\lceil t \rceil_\tau)F^N(z)\|
	&=\|P^N({\rm Id}_H-S(\lceil t \rceil_\tau-t)) S(t) F(z)\| \\
	&\le C (\lceil t \rceil_\tau-t)^{\frac{\tilde{\theta}}{2}} \|S(t) F(z)\|_{\tilde{\theta}} \\
	&\le C \tau^{\frac{\tilde{\theta}}{2}}K_{F_{\tilde{\theta}}}(t)(1+\|z\|).
	\end{align*}
	One can take $\tilde{\theta}\in(0,1)$ to obtain two integrable functions $R_{F^N}(t):=C \lambda_{N+1}^{-\frac{\tilde{\theta}}{2}} K_{F_{\tilde{\theta}}} (t)$ and $R_{F^{N,\tau}}(t):=C \tau^{\frac{\tilde{\theta}}{2}}K_{F_{\tilde{\theta}}}(t)$.
	Therefore, $r_F=\tilde{\theta}<1$ and $\eta_F=\frac{\tilde{\theta}}{2}<\frac12$.
	
	To estimate the diffusion part, we have
	\begin{align*}
	\|S^N(t)G^N(z)-S(t)G(z)\|^2_{\LL^0_2}
	= \sum_{j=N+1}^{\infty}e^{-2\lambda_j t}\|G(z)e_j\|^2
	\le C\sum_{j=N+1}^{\infty}e^{-2\lambda_j t} (1+\|z\|)^2
	\end{align*}
	and 
	\begin{align*}
	\big\|(S^N(t)-S^N(\lceil t \rceil_\tau))G^N(z)\big\|^2_{\LL^0_2}
	&=\sum_{j=1}^N (e^{-\lambda_j t}-e^{-\lambda_j \lceil t \rceil_\tau})^2 
	\|G(z)e_{j}\|^2 \\
	&\le C \sum_{j=1}^\infty (e^{-\lambda_j t}-e^{-\lambda_j \lceil t \rceil_\tau})^2
	(1+\|z\|)^2.
	\end{align*}
	Denoting 
	\begin{align*} 
	R_{G^N}(t):=\left( C\sum_{j=N+1}^{\infty}e^{-2\lambda_j t}   \right)^{\frac12}
	\quad \text{and}\quad 
	R_{G^{N,\tau}}(t):=\bigg(   C \sum_{j=1}^\infty (e^{-\lambda_j t}-e^{-\lambda_j \lceil t \rceil_\tau})^2 \bigg)^\frac12,
	\end{align*} 
	we obtain that 
	\begin{align*}
	\int_0^T R^2_{G^N}(t){\rm d}t&\le C \sum_{j=N+1}^{\infty} j^{-2}\le C  N^{-1},\\
	\int_0^T R^2_{G^{N,\tau}}(t) {\rm d}t& \le C\sum_{j=1}^\infty \lambda_j^{-1} 
	(\tau^2 \lambda_j^{2} \wedge 1) \le C \tau^{\frac12}.
	\end{align*} 
	Consequently, we conclude that $r_G=\frac12$ and $\eta_G=\frac14$.
\end{proof}

\begin{lemma} \label{lm-hol}
	Let $p\ge 2$ and $\beta\in (0,2]$. 
	Suppose that $X_0:\Omega\rightarrow \dot{H}^\beta$ is   $\FFF_0/\mathcal{B}(\dot{H}^\beta)$-measurable such that 
	$X_0\in \mathbb L^p(\Omega;\dot{H}^\beta)$.
	Then Assumption \ref{ap-xn-hol} holds with  $\gamma=\frac\beta2\wedge \frac14$ .
\end{lemma}

\begin{proof}
	Let $0\le s<t\le T$. 
	By the Minkowski and Burkholder--Davis--Gundy inequalities, we get 
	\begin{align*}
	\|X^N(t)-X^N(s)\|_{\mathbb L^p(\Omega;H)}  \le
	& \|(S^N(t)-S^N(s))X^N_0\|_{\mathbb L^p(\Omega;H)} \\
	&+\int_s^t \|S^N(t-r)F^N(X^N(r))\|_{\mathbb L^p(\Omega;H)} {\rm d}r \\
	&+\int_0^s \|(S^N(t-s)-{\rm Id}_H) S^N(s-r) F^N(X^N(r))\|_{\mathbb L^p(\Omega;H)} {\rm d}r \\
	&+ C\bigg(\int_s^t \|S^N(t-r)G^N(X^N(r))\|^2_{\mathbb L^p(\Omega;\LL_2^0)}  {\rm d}r\bigg)^\frac12 \\
	&+C\bigg(\int_0^s \|(S^N(t-s)-{\rm Id}_H) S^N(s-r) G^N(X^N(r))\|^2_{\mathbb L^p(\Omega;\LL_2^0)}   {\rm d}r\bigg)^\frac12\\
	=& :I_1+I_2+I_3+I_4+I_5.
	\end{align*}
	The smoothing effect \eqref{ana} of the semigroup $S$ yields
	\begin{align*}
	I_1 
	&\le C \|(-A)^{-\frac\beta2}(S(t-s)-{\rm Id}_H)\|_{\LL(H)} \|S(s)\|_{\LL(H)} \|X_0\|_{\mathbb L^p(\Omega;\dot{H}^\beta)}\\
	&\le C(t-s)^{\frac\beta2}.
	\end{align*}
	Since $K_F$ and $K_G$ given by \eqref{kf0} and \eqref{kg-whi0}, respectively, satisfy
	\begin{align*} 
	\int_0^t K_F(r) {\rm d}r \le C t^{\frac12}, \quad
	\left(\int_0^t K^2_G(r) {\rm d}r \right)^{\frac12}\le C t^{\frac14},
	\quad t\in (0,T],
	\end{align*}
	we obtain by Assumptions \ref{ap-con}-\ref{ap-spe} that
	\begin{align*}
	I_2
	&\le \bigg(\int_0^{t-s} K_{F^N}(r) {\rm d}r \bigg) 
	\bigg(1+\sup_{t\in [0,T]} \|X^N(t)\|_{\mathbb L^p(\Omega;H)}\bigg) \\
	&\le \bigg(\int_0^{t-s} K_{F}(r) {\rm d}r \bigg) 
	\bigg(1+\sup_{t\in [0,T]} \|X(t)\|_{\mathbb L^p(\Omega;H)}\bigg) \\
	&\le C(t-s)^{\frac12}
	\end{align*}
	and 
	\begin{align*}
	I_4
	&\le C\bigg(\int_0^{t-s} K^2_{G^N}(r) {\rm d}r \bigg)^\frac12 
	\bigg(1+\sup_{t\in [0,T]} \|X^N(t)\|_{\mathbb L^p(\Omega;H)}\bigg) \\
	&\le C\bigg(\int_0^{t-s} K^2_{G}(r) {\rm d}r \bigg)^\frac12 
	\bigg(1+\sup_{t\in [0,T]} \|X(t)\|_{\mathbb L^p(\Omega;H)}\bigg) \\
	&\le C(t-s)^{\frac14}.
	\end{align*}
	For the term $I_3$, we use \eqref{kf} to derive that for any $\tilde{\theta}\in(0,1)$,
	\begin{align*}
	I_3
	&\le \int_0^s \|(-A)^{-\frac{\tilde{\theta}}{2}}(S(t-s)-{\rm Id}_H)\|_{\LL(H)}
	\|S^N(s-r) F^N(X^N(r))\|_{\mathbb L^p(\Omega;\dot{H}^{\tilde{\theta}})} {\rm d}r \\
	&\le  C(t-s)^{\frac{\tilde{\theta}}{2}} \bigg(\int_0^s K_{F_{\tilde{\theta}}} {\rm d}r\bigg)  \Big(1+\sup_{t \in [0,T]}\|X(t)\|_{\mathbb L^p(\Omega;H)}\Big)  
	\le C(t-s)^{\frac{\tilde{\theta}}{2}}.
	\end{align*}
	By the definition of the $\LL_2^0$-norm and the Lipschitz continuity  of $g$, we obtain
	\begin{align*}
	I_5 
	&\le C \bigg(\int_0^s \|(S(t-r)-S(s-r)) G(X^N(r))\|^2_{\mathbb L^p(\Omega;\LL^0_2)}{\rm d}r\bigg)^\frac12  \\
	&\le C \bigg(\int_0^s \sum_{j=1}^{\infty}  
	\left(e^{-\lambda_j(t-r)}-e^{-\lambda_j(s-r)}\right)^2
	\|g(X^N(r)) e_j\|_{\mathbb L^p(\Omega;H)}^2 {\rm d}r \bigg)^\frac12  \\
	&\le C \bigg(\int_0^s \sum_{j=1}^{\infty}  
	\left(e^{-\lambda_j(t-r)}-e^{-\lambda_j(s-r)}\right)^2
	{\rm d}r \bigg)^\frac12  \bigg(1+\sup_{t\in [0,T]} \|X(t)\|_{\mathbb L^p(\Omega;H)}\bigg)\\
	&\le  C\left(\sum_{j=1}^\infty \lambda_j^{-1} 
	((t-s)^2 \lambda_j^{2} \wedge 1)\right)^{\frac12}\\
	&\le C(t-s)^\frac14.
	\end{align*}
	The above estimates show
	\begin{align*}
	\sup_{N\in \nn_+}\sup_{0\le s<t\le T}\frac{\|X^N(t)-X^N(s)\|_{\mathbb L^p(\Omega;H)}}{(t-s)^{\frac{\beta}{2}\wedge \frac14}} \le C,
	\end{align*}
	i.e., Assumption \ref{ap-xn-hol} holds with $\gamma=\frac\beta2\wedge \frac14$.
\end{proof}

Combining Lemmas \ref{ap-whi}-\ref{lm-hol} with Theorem \ref{main}, we obtain the strong convergence rate for scheme \eqref{scheme} applied to Eq. \eqref{she} driven by the space-time white noise.

\begin{theorem}\label{tm-whi}
	Let $p\ge 2$ and $\beta\in (0,2]$. 
	Assume that $X_0:\Omega\rightarrow \dot{H}^\beta$ is   $\FFF_0/\mathcal{B}(\dot{H}^\beta)$-measurable such that 
	$X_0\in \mathbb L^p(\Omega;\dot{H}^\beta)$, then there exists a constant $C$ such that
	\begin{align*}
	\sup_{0\le k \le K}\|X(t_k)-X^N_k\|_{\mathbb L^p(\Omega;H)}\le C (N^{-({\beta\wedge\frac12})}+\tau^{\frac{\beta}{2}\wedge \frac14}),
	\end{align*}
	where $X(t)$ is the unique mild solution of Eq. \eqref{she} and $X^N_k$ is defined by scheme \eqref{scheme}.
\end{theorem}

\begin{remark} \label{rk-whi}
	For the following linear equation with additive space-time white noise
	\begin{align*}
	\begin{split}
	&{\rm d}X(t)=AX(t){\rm d}t+{\rm d}W(t),
	\ t\in (0,T];\\
	& X(0)=0,
	\end{split}
	\end{align*}
	assume that $H=U=\mathbb{L}^2(0,1)$ and that $A=\Delta$ is the Dirichlet Laplacian, then
	the unique mild solution is given by the Ornstein--Uhlenbeck process
	$W_A(t)=\int_{0}^{t}S(t-r){\rm d}W(r)$.
	We have that
	\begin{align*}
	\|W_A(t)\|^2_{\mathbb L^2(\Omega;\dot{H}^\theta)}&=\int_{0}^{t}\|(-A)^{\frac{\theta}{2}}S(t-r)\|_{\LL_2^0}^2{\rm d}r\\
	&=\int_{0}^{t}\sum_{j=1}^{\infty}\|(-A)^{\frac{\theta}{2}}S(t-r)e_j\|^2{\rm d}r\\
	&=\sum_{j=1}^{\infty} \lambda_j^{\theta}\frac{1-e^{-2\lambda_j t}}{2\lambda_j}\\
	&\ge \frac{1-e^{-2t}}{2}\sum_{j=1}^{\infty} \lambda_j^{\theta-1},\quad t>0,
	\end{align*}   
	which is divergent as long as $\theta\ge\frac12$. This implies that the optimal spatial regularity is $\dot{H}^\theta$ with $\theta<\frac12$. Comparing this with the best probable convergence order $\frac12$ provided in Theorem \ref{tm-whi}, it reveals that 
	the spectral Galerkin approximation is superconvergent in the spacial direction.
	
	Moreover, since the optimal H\"older index of the solution under the $\mathbb L^p(\Omega;H)$-norm proved in Lemma \ref{lm-hol} is $\frac14$, the temporal convergence rate of the exponential integrator is optimal.

\end{remark}

\section{Numerical Experiments}
\label{sec5}

In this section, we give numerical tests to verify our theoretical result on the convergence rate of the fully discrete scheme \eqref{scheme}. 
We apply scheme \eqref{scheme} to the following one-dimensional stochastic advection-diffusion-reaction equation (where we take $f(v)=-v$, $\widetilde f(v)=-v/(1+|v|)$ and $g(v)=(1+v)/8$ for $v\in \rr$):
\begin{align}\label{ex}
\begin{split}
{\rm d} X(t,\xi)
=\Big(\Delta X(t,\xi)-\nabla X(t,\xi)-\frac{X(t,\xi)}{1+|X(t,\xi)|}\Big) {\rm d}t+
\frac{1+X(t,\xi)}{8} {\rm d}W(t,\xi),
\end{split}
\end{align}
with homogeneous Dirichlet boundary condition, in the time-space domain $(0,T]\times \OOO=(0,1] \times (0,1)$. We refer to \cite{BCD06(RG)} and references therein for relevant applications.

We take the initial datum of the form
\begin{align}\label{ini}
X(0,\xi)=X_0(\xi)=\sum_{j=1}^\infty \frac{e_j(\xi)}{j^{1.01}},
\quad \xi\in (0,1).
\end{align}
Besides,
\begin{align}\label{W}
W(t,\xi)=\sum_{j=1}^\infty e_j(\xi) \beta_j(t),
\quad (t,\xi)\in (0,1] \times (0,1).
\end{align} 
Here 
$\{e_j(\cdot)=\sqrt{2}\sin(j\pi\cdot)\}_{j\in \nn_+}$ is a sequence of eigenfunctions of the Laplacian operator $\Delta$, which forms an orthonormal basis of $\mathbb L^2(0,1)$,
and $\{\beta_j\}_{j\in \nn_+}$ is a sequence of independent standard Brownian motions.
Then the solution of \eqref{ex} with initial datum \eqref{ini} and cylindrical Wiener process \eqref{W} possesses a unique mild solution in $\CC([0,T]; \mathbb L^p(\Omega; \dot{H}^\theta))$ for any $p\ge 2$ and $\theta<\frac12$.


To simulate the `exact' solution, we perform the fully discrete scheme by $N=2^{9}$ for the dimension of spectral Galerkin approximation and  $\tau=2^{-13}$ for the time step size of the exponential integrator. The series of \eqref{W} is truncated as a finite summation up to $2^{9}$ terms.

Taking $\tau=2^{-i}$ with $i=7,8,9,10$ and $N=2^{9}$, we have the mean-square convergence rate  in temporal direction shown in the left picture of Figure \ref{f1}. Choosing four different spatial dimensions $N=2^{i}$ with $i=4,5,6,7$ and fixing $\tau=2^{-13}$, we get the mean-square convergence rate in spatial direction, which are presented in the right picture of Figure \ref{f1}. The expectation is approximated from the average of $200$ sample paths. 

From these numerical experiments, it is clear that the spatial and temporal strong convergence orders of the scheme \eqref{scheme} applied to Eq. \eqref{ex} are $\frac12$ and $\frac14$, respectively.
They validate the theoretical result of Theorem \ref{main}.

\begin{figure}
	\centering
	\subfigure[Temporal error]{
		\begin{minipage}[t]{0.46\linewidth}
			\includegraphics[height=5cm,width=5.8cm]{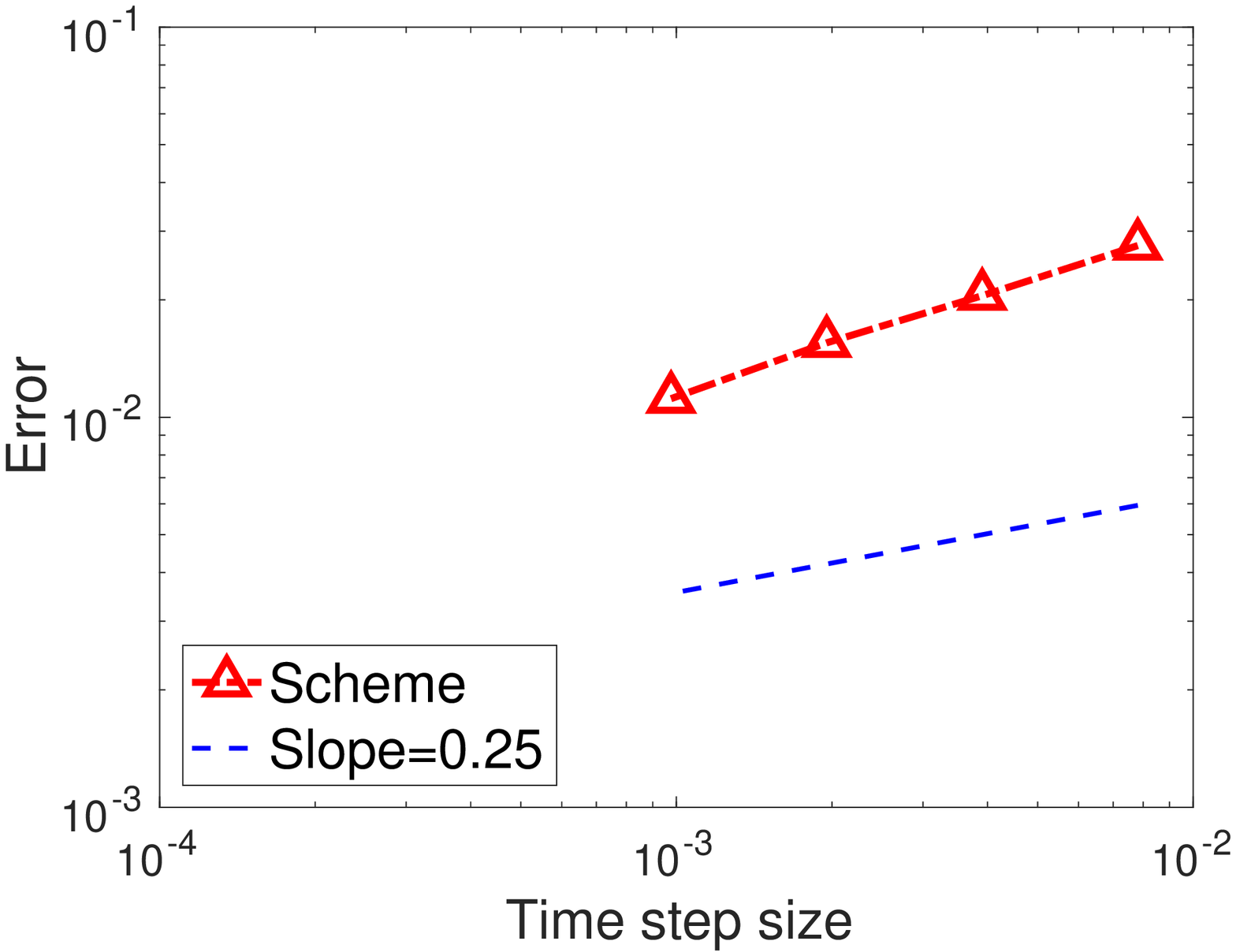}
		\end{minipage}
	}
	\subfigure[Spatial error]{
		\begin{minipage}[t]{0.46\linewidth}
			\includegraphics[height=5cm,width=5.8cm]{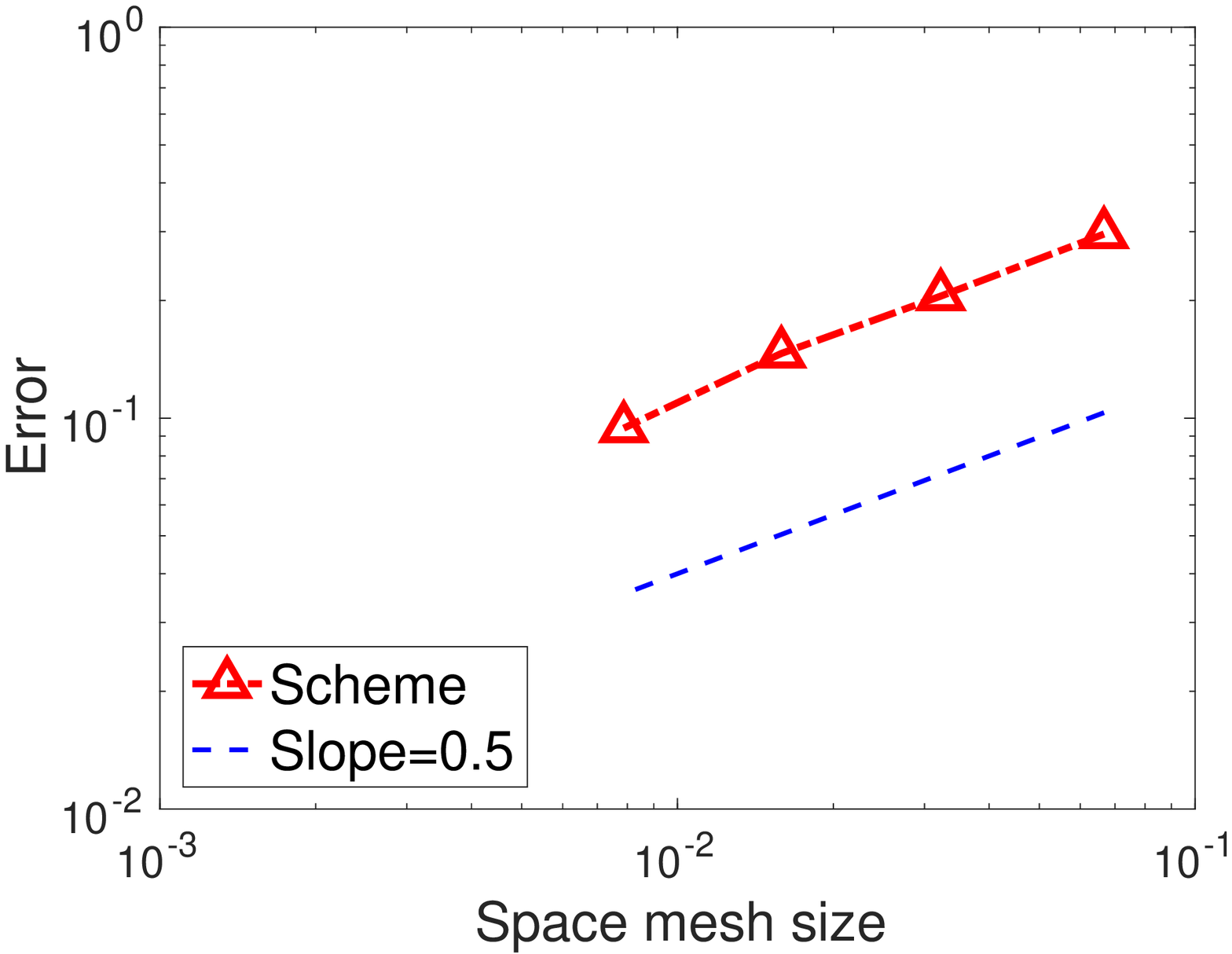}
		\end{minipage}
	}
	\caption{Temporal (left) and spatial (right) convergence rates.} \label{f1}
\end{figure}




%
%



\end{document}